\documentclass[10pt]{amsart}
\usepackage{amsmath}
\usepackage{xypic}
\usepackage{amssymb}
\usepackage{amsfonts}
\usepackage{amsthm}
\usepackage{enumerate}
\usepackage[all]{xy}
\usepackage[latin1]{inputenc}
\usepackage{graphicx}
\usepackage{latexsym}
\usepackage{color}
\usepackage{accents}
\usepackage{mathdots}
\usepackage{mathrsfs}
\input xy

\usepackage{tikz}

\makeatletter
\newtheoremstyle{definition}
{5pt}
{3pt}
{}
{0pt}
{\scshape}
{.}
{5pt}
{\thmname{#1} \thmnumber{#2} \thmnote{[#3]}} 

\newtheoremstyle{theorems}
{5pt}
{3pt}
{\itshape}
{0pt}
{\scshape}
{.}
{5pt}
{\thmname{#1} \thmnumber{#2}\thmnote{[#3]}} 

\swapnumbers \theoremstyle{theorems}

\newtheorem{theorem}{Theorem}[section]
\newtheorem{lemma}[theorem]{Lemma}
\newtheorem{definition-theorem}[theorem]{Definition-Theorem}
\newtheorem{proposition}[theorem]{Proposition}
\newtheorem{corollary}[theorem]{Corollary}
\newtheorem{definition}[theorem]{Definition}

\newtheorem{remark}[theorem]{Remark}
\newtheorem{notation}[theorem]{Notation}

\newtheorem{notation-remark}[theorem]{Notation-Remarks}
\newtheorem{conjecture}[theorem]{Conjecture}

\begin{document}

\title[On the monotonicity of the generalized Markov numbers]{\sc On the monotonicity of the generalized Markov numbers}


\author[Min Huang]{Min Huang \hspace{-3pt}}



\subjclass[2010]{}



\address{Min Huang \\ School of Mathematics (Zhuhai), Sun Yat-sen University, Zhuhai, China.}
\email{huangm97@mail.sysu.edu.cn}


\maketitle

\begin{abstract}
Using the Markov distance and Ptolemy inequality introduced by Lee-Li-Rabideau-Schiffler \cite{LLRS}, we completely determine the monotonicity of the generalized Markov numbers along the lines of a given slope.
\end{abstract}

\medskip

 \tableofcontents

\section{Introduction}

A \emph{Markov number} is any number in the triple $(x,y,z)$ of positive integer solutions to the Diophantine equation
$$x^2+y^2+z^2=3xyz,$$ known as the \emph{Markov equation}.

Every Markov number appears as the maximum of some Markov triple, The \emph{Markov Uniqueness Conjecture} by Frobenius from 1913 asserts that each Markov number appears as the maximum of a unique Markov triple \cite{F,A,LLRS,J,RS}.

As an approach to studying the Uniqueness Conjecture, Aigner \cite{A} proposed three conjectures, called fixed numerator, fixed denominator, and fixed sum conjectures, which say that the Markov numbers increase along the lines $y$-axis, $x$-axis, and $y=x$, respectively.

Propp \cite{P} and Beineke-Br\"{u}stle-Hille \cite{BBH} found that the Markov numbers are the specialized cluster variables of the once-punctured torus cluster algebras. The family of cluster algebras from surfaces, introduced by Fomin-Shapiro-Thurston \cite{FST} is a class of important and special cluster algebras. The cluster variables can be computed in terms of the perfect matching of certain snake graphs \cite{MS,MSW,H2}, see also for the quantum case in \cite{CL,H1,H3}. In \cite{MSW1}, for a given cluster algebra $\mathcal A$ from a marked surface $(S,M)$, Musiker-Schiffler-Williams associate any generalized curve $\gamma$ on $(S,M)$ with an element $x_{\gamma}$.

To study the ordering of the Markov numbers and the Uniqueness Conjecture, Lee-Li-Rabideau-Schiffler \cite{LLRS} introduced the Markov distance on the plane.
By comparing the Markov tree and the Farey tree, one sees that the Markov numbers can be indexed by the rational numbers between zero and one, equivalently, the set $\{(q,p)\in \mathbb Z_{>0}^2\mid p<q, g.c.d.(p,q)=1\}$, see \cite{A}.
Using the Markov distance, Lee-Li-Rabideau-Schiffler extended the Markov numbers to the numbers indexed by all $(q,p)\in \mathbb Z_{>0}^2$ with $p<q$, we call them \emph{generalized Markov numbers} in this paper. They show that the Markov numbers increase and decrease along the line with some special slopes. They also show that there are some lines such that the Markov numbers are not monotonic along these lines.

Using hyperbolic geometry, Gaster provides the boundary slopes for which the Markov numbers decrease and increase, respectively. In this paper, we consider the monotonicity of generalized Markov numbers and give a parallel result to that of Gaster.

Our main result is the following, which was conjectured by Lee-Li-Rabideau-Schiffler \cite{LLRS} \footnote{Thanks to the authors of \cite{LLRS} for sharing their revised manuscript in private communication. See also \cite{J}.}.

\begin{theorem}(Theorem \ref{thm-mono}, Proposition \ref{prop-mon} (2))

\begin{enumerate}[$(1)$]
  \item For $k\in \mathbb Q$ with $k\geq -\frac{ln\frac{3(2+\sqrt 2)}{4}}{ln\frac{2(2+\sqrt 2)}{3}}\approx -1.1432$, the generalized Markov numbers increase with $x$ along any line $l: y=kx+b$;
  \item For $k\in \mathbb Q$ with $k\leq -\frac{2ln(\frac{1+\sqrt{5}}{2})}{ln(\frac{3(1+\sqrt{5})}{2\sqrt{5}})}\approx -1.2417$, the generalized Markov numbers decrease with $x$ along any line $y=kx+b$;
  \item For any $k\in \mathbb Q$ with $-\frac{2ln(\frac{1+\sqrt{5}}{2})}{ln(\frac{3(1+\sqrt{5})}{2\sqrt{5}})}<k<-\frac{ln\frac{3(2+\sqrt 2)}{4}}{ln\frac{2(2+\sqrt 2)}{3}}$, then for almost all $b\in \mathbb Q$, the generalized Markov numbers are not monotonic along the line $y=kx+b$.
  \item For the lines along which the generalized Markov numbers are not monotonic, when the $x$-coordinate increases the generalized Markov numbers first decrease and then increase.
  \item the generalized Markov numbers are not monotonic along the line $l$ if and only if $r_l(u_l,v_l)<1$ and $r_l(z_l,w_l)> 1$.
\end{enumerate}

\end{theorem}

\begin{remark}
Theorem 1.1(1) solves \cite[Conjecture 6.11]{LLRS}, Theorem 1.1(2) solves \cite[Conjecture 6.8]{LLRS} and Theorem 1.1(3) (4) solve \cite[Conjecture 6.12]{LLRS}.
\end{remark}

Note that the Markov distance does not have triangle inequality, we find the following interesting inequality, as modified triangle inequality.

\begin{proposition}(Proposition \ref{cor-t1})
Let $(x,y), (x',y'),(x'',y'')\in \mathbb Z_{\geq 0}^2$ be three points with $x\geq y,x'\geq y'$ and $x''\geq y''$. If $(x',y')=\frac{(x,y)+(x'',y'')}{2}$ then
$$m_{x,y}+m_{x'',y''}\geq 2m_{x',y'}.$$
\end{proposition}

The structure of this paper is the following. Section \ref{sec-p} is preliminary, we review the Markov distance, generalized Markov numbers defined by Lee-Li-Rabideau-Schiffler, and some properties herein. We introduce the ratios between two generalized Markov numbers and study the monotonicity in Section \ref{sec-3}. We study the monotonicity of the generalized Markov numbers and give the proof of the main result in Section \ref{sec-4}.

\medskip

{\bf Convention:} (i)\; The points that appeared in the paper are always assumed to lie in the area $\{(x,y)\in \mathbb Z^{2}_{>0}\mid x>y\}$ unless otherwise
stated. (ii)\; When we consider the monotonicity of the generalized Markov numbers along a line $l$, we always assume that $l\cap \{(x,y)\in \mathbb Z^{2}_{>0}\mid x>y\}\neq \emptyset$.

\section{Preliminary}\label{sec-p}

In this section, we review the Markov distance, generalized Markov numbers defined in \cite{LLRS}, and some properties.

Let $\gamma$ be a curve connecting two  points but not passing through a third one. Musiker-Schiffler-Williams \cite{MSW1} associated $\gamma$ with an element $x_{\gamma}$ in the once-punctured torus cluster algebra. Denote $|\gamma|=x_{\gamma}|_{x_1,x_2,x_3}=1$ the positive integer obtained from $x_\gamma$ by specializing the initial cluster algebras $x_1,x_2$ and $x_3$ to $1$.

\begin{definition}\label{def}\cite{LLRS}
For any points $A,B\in \mathbb Z^2$, the \emph{Markov distance} $|AB|$ between $A$ and $B$ is defined as $|\gamma_{AB}^L|$, where $\gamma_{AB}^L$ is the left deformation of $AB$.
\end{definition}

By the skein relation \cite{MW} of cluster algebras from surfaces, the Markov distance has the following important property.

\begin{theorem}\cite[Corollary 3.6]{LLRS}
(Ptolemy inequality) Given four points $A,B,C,D$ in the plane such that the straight line segments $AB,BC,CD,DA$ form a convex quadrilateral with diagonals $AC$ and $BD$,
we have
$$|AC| |BD|\geq |AB| |CD| + |AD| |BC|.$$
\end{theorem}

For any $A=(x,y)\in \mathbb Z^2$, denote $m_{x,y}=|OA|$.

\begin{definition}
For any $(x,y)\in \mathbb Z_{>0}^2$ such that $x>y$, we call $m_{x,y}$ a \emph{generalized Markov numbers}.
\end{definition}

Note that if $x$ and $y$ are coprime then $m_{x,y}=m_{\frac{x}{y}}$ is the usual Markov number.

We now list a few results on the generalized Markov numbers that we will need later.

Recall that the Fibonacci numbers are $\{\mathcal F_n\mid n\geq 0\}$ which satisfies $\mathcal F_0=0, \mathcal F_1=\mathcal F_2=1, \mathcal F_n=\mathcal F_{n-1}+\mathcal F_{n-2}$ for $n\geq 2$. It is well-known that the Markov numbers indexed by $(q,1),q>1$ are the odd Fibonacci numbers $\mathcal F_{2q+1}, q>1$.

\begin{equation}
m_{q,1}=\mathcal F_{2q+1}=(\phi^{2q+1}+\phi^{-2q-1})/\sqrt{5}\sim \phi^{2q+1}/\sqrt{5},
\end{equation}
where $\phi=\frac{\sqrt{5}+1}{2}$.

Recall that the Pell numbers are $\{P_n\mid n\geq 0\}$ which satisfies $P_0=0, P_1=1, P_2=1, P_n=2P_{n-1}+P_{n-2}$ for $n\geq 2$. It is well-known that the Markov numbers indexed by $(q,q-1),q>1$ are the odd Pell numbers $P_{2q-1}, q>1$.

\begin{equation}
m_{q,q-1}=P_{2q-1}=\frac{(1+\sqrt{2})^{2q-1}-(1-\sqrt{2})^{2q-1}}{2\sqrt{2}}\sim \frac{(1+\sqrt{2})^{2q-1}}{2\sqrt{2}}.
\end{equation}

The following useful lemma is given in \cite{LLRS}.

\begin{lemma}\cite[Lemma 6.2]{LLRS}\label{lem-fn}
Let $p,q$ be coprime positive integers and let $f_n=m_{nq,np}$. Thus $f_0=0$ and $f_1=m_{q,p}$
is the Markov number. Then, for $n\geq 2$,
$$f_n=3f_1f_{n-1}-f_{n-2}.$$
As a consequence, we have $f_n=\frac{f_1}{\sqrt{9f_1^2-4}}(\alpha^n-\alpha^{-n})$, where $\alpha=(3f_1+\sqrt{9f_1^2-4})/2$.
\end{lemma}

In particular, we have

\begin{proposition}\cite{LLRS}\label{prop-app}
For any $q, n>1$, we have
\begin{enumerate}[$(1)$]
  \item
  \begin{equation*}
  m_{qn,n}= \frac{m_{q,1}}{\sqrt{9m_{q,1}^2-4}}(\alpha^n-\alpha^{-n})\sim 3^{n-1}m_{q,1}^n \hspace{5mm} (q\to \infty),
  \end{equation*}
  \item
  \begin{equation*}
   m_{qn,(q-1)n}= \frac{m_{q,q-1}}{\sqrt{9m_{q,q-1}^2-4}}(\alpha^n-\alpha^{-n})\sim 3^{n-1}m_{q,q-1}^n \hspace{5mm} (q\to \infty),
\end{equation*}
\end{enumerate}
where $\alpha=\frac{3m_{q,1}+\sqrt{9m_{q,1}^2-4}}{2}$.

\end{proposition}

As a corollary of Lemma \ref{lem-fn}, we have the following observation.

\begin{lemma}\label{lem-m}
For any $q\in \mathbb Z_{\geq 0}$, we have
\begin{enumerate}[$(1)$]
\item $m_{p,0}=\mathcal F_{2q}$,
\item $m_{p,p}=P_{2q}$.
\end{enumerate}
\end{lemma}

\begin{proof}
We shall only prove (1) as (2) can be proved similarly. For any $q\geq 2$, as $\mathcal F_{2q-2}=\mathcal F_{2q-3}+\mathcal F_{2q-4}$, we have that
$\mathcal F_{2q-2}+\mathcal F_{2q-3}=2\mathcal F_{2q-2}-\mathcal F_{2q-4}$, this implies
$$\mathcal F_{2q-1}=2\mathcal F_{2q-2}-\mathcal F_{2q-4},$$
Thus we have
$$\mathcal F_{2q}=3\mathcal F_{2q-2}-\mathcal F_{2q-4}.$$
Moreover as $m_{0,0}=0=\mathcal F_0$ and $m_{1,0}=1=\mathcal F_2$, by Lemma \ref{lem-fn} the result follows.

\end{proof}

\section{Ratio between two generalized Markov numbers}\label{sec-3}

\subsection{Horizontal and vertical ratios}

\begin{definition}
For any $(q,p)\in \mathbb \mathbb Z_{\geq 0}^2$ with $p\leq q$, the \emph{horizontal ratio} at $(q,p)$ is the ratio $\frac{m_{q+1,p}}{m_{q,p}}$, denote by $h(q,p)$. If $p<q$, the \emph{vertical ratio} at $(q,p)$ is the ratio $\frac{m_{q,p+1}}{m_{q,p}}$, denote by $v(q,p).$
\end{definition}

Thus $m_{q+1,p}>m_{q,p}$ if and only if $h(q,p)>1$, $m_{q,p+1}>m_{q,p}$ if and only if $v(q,p)>1$.

We first investigate the monotonicity of $h(q,p)$ and $v(q,p)$.

\begin{lemma}\label{lem-h1}
For any $(q,p)\in \mathbb Z_{>0}^2$ with $p\leq q$, we have
\begin{enumerate}
\item $$h(q+1,p)> h(q,p),$$
\item if $p+1\leq q$ then
$$h(q,p)>h(q,p+1).$$
\end{enumerate}
\end{lemma}

\begin{proof}
(1) Let $O=(0,0), A_1=(q,p), A_2=(q+1,p), A_3=(q+2,p)$. Let $O'=(1,0)$. By the Ptolemy inequality, we have $|OA_3||O'A_2|>|OA_2||O'A_3|$, that is
$$|OA_3||OA_1|>|OA_2|^2.$$
It follows that $\frac{|OA_3|}{|OA_2|}>\frac{|OA_2|}{|OA_1|},$ that is
 $$h(q+1,p)>h(q,p).$$

(2) Let $O=(0,0), A_1=(q,p), A_2=(q+1,p), A_1'=(q,p+1), A_2'=(q+1,p+1)$. Let $O'=(0,1)$. By the Ptolemy inequality, we have $|OA_1'||O'A'_2|>|O'A_1'||OA_2'|$, that is
$$|OA'_1||OA_2|>|OA_1||OA_2'|.$$
It follows that $\frac{|OA_2|}{|OA_1|}>\frac{|OA'_2|}{|OA'_1|},$ that is
$$h(q,p)>h(q,p+1)$$

\begin{figure}[h]
\includegraphics{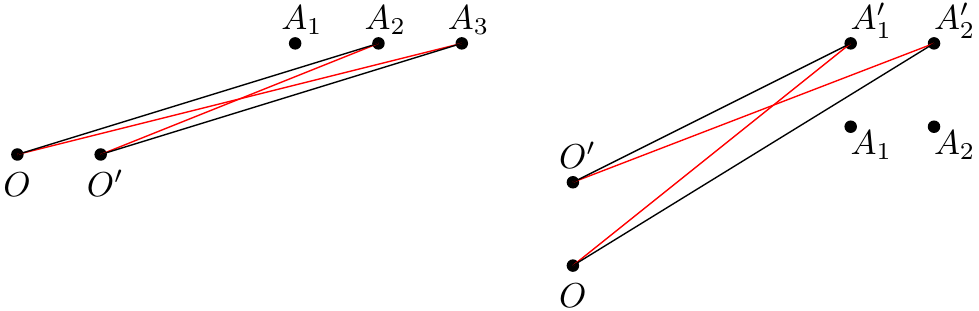}

{\rm Figure for Lemma \ref{lem-h1}}
\end{figure}

\end{proof}

\begin{remark}
 Lemma \ref{lem-h1} implies the function $h(x,y)$ is increasing along the $x$-axis and decreasing along the line $y$-axis.
\end{remark}

\begin{corollary}\label{cor-h1}
For any $(q,p)\in \mathbb Z_{>0}^2$ with $p\leq q$, we have
$$\frac{(1+\sqrt 2)^{4q+2}+1}{(1+\sqrt 2)((1+\sqrt 2)^{4q}-1)}\leq h(q,p)\leq \frac{\phi^{4q+6}+1}{\phi^2(\phi^{4q+2}+1)},$$
where $\phi=\frac{\sqrt 5+1}{2}$. In particular, we have
$$(1+\sqrt 2)m_{(q,p)}<m_{(q,p+1)}<\frac{3+\sqrt 5}{2}m_{q,p+1}.$$
\end{corollary}

\begin{proof}
By Lemma \ref{lem-h1}, we have $h(q,q)\leq h(q,p)\leq h(q,1)$. By Lemma \ref{lem-m}, we have
$$h(q,q)=\frac{m_{q+1,q}}{m_{q,q}}=\frac{P_{2q+1}}{P_{2q}}=\frac{(1+\sqrt 2)^{4q+2}+1}{(1+\sqrt 2)((1+\sqrt 2)^{4q}-1)}.$$
$$h(q,1)=\frac{m_{q+1,1}}{m_{q,1}}=\frac{\mathcal F_{2q+3}}{\mathcal F_{2q+1}}=\frac{\phi^{4q+6}+1}{\phi^2(\phi^{4q+2}+1)}.$$

As $\frac{(1+\sqrt 2)^{4q+2}+1}{(1+\sqrt 2)((1+\sqrt 2)^{4q}-1)}>1+\sqrt 2$ and $\frac{\phi^{4q+6}+1}{\phi^2(\phi^{4q+2}+1)}<\phi^2=\frac{3+\sqrt 5}{2}$, we have
$$(1+\sqrt 2)m_{(q,p)}<m_{(q,p+1)}<\frac{3+\sqrt 5}{2}m_{q,p+1}.$$

\end{proof}

\begin{lemma}\label{lem-v1}
For any $(q,p)\in \mathbb Z_{\geq 0}^2$ with $p<q$, we have
\begin{enumerate}[$(1)$]
\item if $p+1<q$ then
$$v(q,p+1)>v(q,p),$$
\item
$$v(q,p)<v(q+1,p).$$
\end{enumerate}
\end{lemma}

\begin{proof}
(1) Let $O=(0,0), A_1=(q,p), A_2=(q,p+1), A_3=(q,p+2)$. Let $O'=(0,1)$. By the Ptolemy inequality, we have $|OA_3||O'A_2|>|OA_2||O'A_3|$, that is
$$|OA_3||OA_1|>|OA_2|^2.$$
It follows that $\frac{|OA_3|}{|OA_2|}>\frac{|OA_2|}{|OA_1|},$ that is
 $$v(q,p+1)>v(q,p).$$

(2) Let $O=(0,0), A_1=(q,p), A_2=(q+1,p), A_1'=(q,p+1), A_2'=(q+1,p+1)$. Let $O'=(0,1)$. By the Ptolemy inequality, we have $|OA_2||O'A'_2|>|O'A_2||OA_2'|$, that is
$$|OA_2||OA'_1|>|OA_1||OA_2'|.$$
It follows that $\frac{|OA_2|}{|OA_1|}>\frac{|OA'_2|}{|OA'_1|},$ that is
$$v(q,p)>v(q+1,p)$$

\begin{figure}[h]
\includegraphics{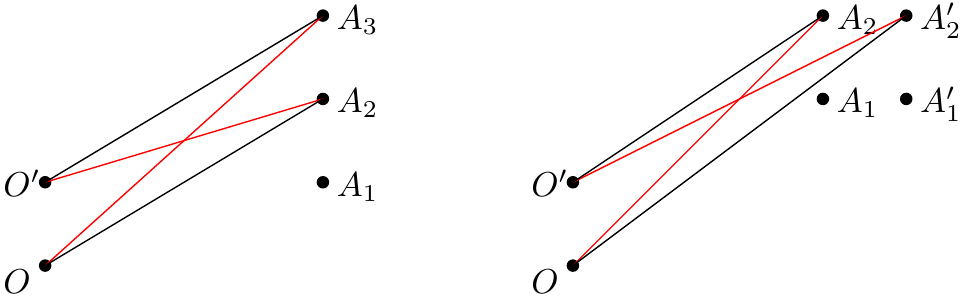}

{\rm Figure for Lemma \ref{lem-v1}}
\end{figure}

\end{proof}

\begin{remark}
 Lemma \ref{lem-v1} implies the function $v(x,y)$ is increasing along the $x$-axis and decreasing along the $y$-axis.
\end{remark}

\begin{corollary}\label{cor-v1}
For any $(q,p)\in \mathbb Z_{\geq 0}^2$ with $p<q$, we have
$$\frac{(1+\sqrt 2)^{4p+4}-1}{(1+\sqrt 2)((1+\sqrt 2)^{4p+2}+1)}\leq v(q,p)\leq \frac{\phi^{4q+2}+1}{\phi(\phi^{4q}+1)},$$
where $\phi=\frac{\sqrt 5+1}{2}$. In particular,
$$m_{q,p}<m_{q,p+1}<\phi m_{q,p}.$$
\end{corollary}

\begin{proof}
By Lemma \ref{lem-v1}, we have $v(p+1,p), v(q,q-1)\leq v(q,p)\leq v(q,0)$. By Lemma \ref{lem-m}, we have
$$v(p+1,p)=\frac{m_{p+1,p+1}}{m_{p+1,p}}=\frac{P_{2p+2}}{P_{2p+1}}=\frac{(1+\sqrt 2)^{4p+4}-1}{(1+\sqrt 2)((1+\sqrt 2)^{4p+2}+1)}.$$
$$v(q,0)=\frac{m_{q,1}}{m_{q,0}}=\frac{\mathcal F_{2q+1}}{\mathcal F_{2q}}=\frac{\phi^{4q+2}+1}{\phi(\phi^{4q}+1)}.$$
The result follows.

\end{proof}

\subsection{Ratio along any line}

Inspirited by the horizontal and vertical ratio, we now define the ratio along any line $l: y=kx+b$ where $k,b\in \mathbb Q$.

For any $t\in \mathbb R$ denote by $l[t]: y=k(x-t)+b$ the line obtained from $l$ by shift along $x$-axis by $t$, by $l\langle t\rangle: (y-t)=k(x-t)+b$ the line obtained from $l$ by shift along $y=x$ by $\sqrt{2}t$. We always assume that $t\in \mathbb Z_{\geq 0}$ unless otherwise
stated. Note that $l\langle t\rangle =l[t-t/k]$ when $k\neq 0$.

\medskip

\emph{List the integral  points in $\{(q,p)\in \mathbb Z_{>0}^2 \mid p< q\}\cap l$ as $(x_1,y_1),(x_2,y_2),\cdots$ in order such that $x_1<x_2<\cdots$.}

\medskip

\begin{notation}\label{no-1}
Denote by $(u_l,v_l),(u'_l,v'_l)$ the first two integral  points in $\{(q,p)\in \mathbb Z_{>0}^2 \mid p< q\}\cap l$. If $k<0$, denote by $(z_l,w_l),(z'_l,w'_l)$ the last two integral  points in $\{(q,p)\in \mathbb Z_{>0}^2 \mid p< q\}\cap l$.
\end{notation}

\medskip

The following follows immediately by the construction of $l[t]$ and $l\langle t\rangle$.

\begin{lemma}\label{lem-fl}
Assume that the line $l$ has negative slope. Then for any $t\in \mathbb Z_{>0}$, we have
\begin{enumerate}[$(1)$]
  \item $(u_{l\langle t\rangle}, v_{l\langle t\rangle})=(u_l+t,v_l+t)$ and $(u'_{l\langle t\rangle}, v'_{l\langle t\rangle})=(u'_l+t,v'_l+t);$
  \item $(z_{l\langle t\rangle}, w_{l\langle t\rangle})=(z_l+t,w_l)$ and $(z'_{l\langle t\rangle}, w'_{l\langle t\rangle})=(z'_l+t,w'_l).$
\end{enumerate}
\end{lemma}

\begin{definition}
Let $k,b\in \mathbb Q$ and $l: y=kx+b$ be a line in $\mathbb R^2$. Let $(x_i,y_i),(x_{i+1},y_{i+1})$ be consecutive points on $l$. The \emph{ratio $r_l(x_i,y_i)$ along $l$} at $(x_i,y_i)$ is defined to be
$$r_l(x_i,y_i):=\frac{m_{(x_{i+1},y_{i+1})}}{m_{(x_i,y_i)}}.$$
\end{definition}

\begin{remark}\label{rem-mo}
By the definition,
\begin{enumerate}[$(1)$]
\item the generalized Markov numbers increase with $x$ along the line $l$ if and only if $r_l(x_i,y_i)\geq 1$ for any $i$.
\item the generalized Markov numbers decrease with $x$ along the line $l$ if and only if $r_l(x_i,y_i)\leq 1$ for any $i$.
\end{enumerate}

\end{remark}

\begin{proposition}\label{prop-l1}
With the foregoing notation. Let $k,b\in \mathbb Q$ and $l: y=kx+b$. Let $(x_i,y_{i}), (x_{i+1},y_{i+1}), (x_{i+2},y_{i+2})$ be three consecutive points on $l$.
\begin{enumerate}[$(1)$]
\item If $b\neq 0$ then
$$r_l(x_{i+1},y_{i+1})>r_1(x_i,y_i),$$
that is, the ratios along $l$ are increase with $x$.
\item If $b=0$ then
$$r_l(x_{i+1},y_{i+1})<r_1(x_i,y_i),$$
that is, the ratios along $l$ are decrease with $x$.
\end{enumerate}
\end{proposition}

\begin{proof}
Let $O=(0,0), A_1=(x_i,y_i), A_2=(x_{i+1},y_{i+1})$ and $A_3=(x_{i+2},y_{i+2})$. Let $O'=(x_{i+1}-x_i,y_{i+1}-y_i)$.

(1) If $b\neq 0$, then $O$ is not on $l$, and thus $OA_3$ crosses $O'A_2$. Then by the Ptolemy inequality, we have
$|OA_3||O'A_2|>|OA_2||O'A_3|$, that is,
$$|OA_3||OA_1|>|OA_2||OA_2|.$$
Thus, $$r_l(x_{i+1},y_{i+1})=\frac{|OA_3|}{|OA_2|}<\frac{|OA_2|}{|OA_1|}=r_l(x_{i},y_{i}).$$

(2) If $b=0$ we have $k>0$. Then $O\in l$ and $\gamma^L_{OA_2}$ crosses $\gamma^L_{O'A_3}$, where $\gamma^L_{OA_2}$ and $\gamma^L_{O'A_3}$ are the left deformations of $OA_2$ and $O'A_3$, respectively, given in \cite{LLRS}. By the Ptolemy inequality, we have
$|OA_2||O'A_3|>|OA_3||O'A_2|$, that is
$$|OA_2||OA_2|>|OA_3||OA_1|.$$
Thus, $$r_l(x_{i},y_{i})=\frac{|OA_2|}{|OA_1|}>\frac{|OA_3|}{|OA_2|}=r_l(x_{i+1},y_{i+1}).$$

\begin{figure}[h]
\includegraphics{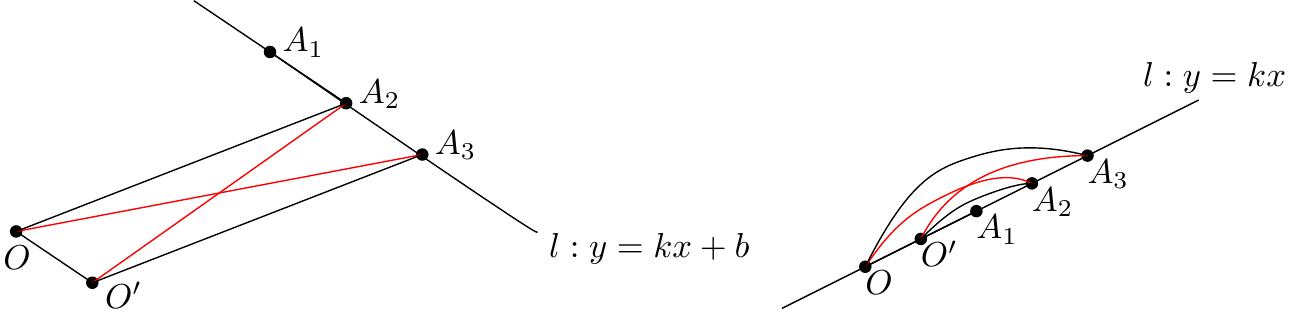}

{\rm Figure for Proposition \ref{prop-l1}}
\end{figure}

\end{proof}

\begin{proposition}\label{prop-mon}
With the foregoing notation. Let $k,b\in \mathbb Q$ and $l: y=kx+b$.
\begin{enumerate}[$(1)$]
  \item If $k\geq 0$ then the generalized Markov numbers increase with $x$ along $l$;
  \item If $k<0$ then the generalized Markov numbers first decrease with $x$ to some point $(x_{i(l)},y_{i(l)})$ then increase with $x$ along $l$; moreover,
  \begin{enumerate}[$(2.1)$]
  \item the generalized Markov numbers increase with $x$ along $l$ if and only if
  $$m_{u_l,v_l}\leq m_{u'_l,v'_l},$$ in terms of Notation \ref{no-1}, that is $r_l(u_l,v_l)\geq 1$;
  \item the generalized Markov numbers decrease with $x$ along $l$ if and only if
  $$m_{z_l,w_l}\geq m_{z'_l,w'_l},$$ in terms of Notation \ref{no-1}, that is $r_l(z_l,w_l)\leq 1$.
  \item the generalized Markov numbers are not monotonic along the line $l$ if and only if $r_l(u_l,v_l)<1$ and $r_l(z_l,w_l)> 1$
  \end{enumerate}

\end{enumerate}

\end{proposition}

\begin{proof}
It follows immediately by Proposition \ref{prop-l1} and Remark \ref{rem-mo}.

\end{proof}

As a corollary of Proposition \ref{prop-l1}, the following can be viewed as modified triangle inequality for the Markov distance.

\begin{proposition}\label{cor-t1}
Let $(x,y), (x',y'),(x'',y'')\in \mathbb Z_{\geq 0}^2$ be three points with $x\geq y,x'\geq y'$ and $x''\geq y''$. If $(x',y')=\frac{(x,y)+(x'',y'')}{2}$ then
$$m_{x,y}+m_{x'',y''}\geq 2m_{x',y'}.$$
\end{proposition}

\begin{proof}

If $x=x'=x''$ we may assume that $y<y'<y''$, by Lemma \ref{lem-v1} we have $v(x,y)< v(x',y')$. Thus
$$\begin{array}{rcl}
m_{x,y}+m_{x'',y''}&=&(\frac{1}{v(x,y)}+v(x',y'))m_{x',y'}\vspace{1.5mm}\\
&> &
(\frac{1}{v(x,y)}+v(x,y))m_{x',y'} \vspace{1.5mm}\\
&\geq &
2m_{x',y'}.
\end{array}$$

If $x\neq x'$ we may assume that $(x,y), (x',y'),(x'',y'')$ are on the line $l: y=kx+b$. We may further assume that $x<x'<x''$.

If $b\neq 0$, by Proposition \ref{prop-l1} repeatedly, we have $\frac{m_{x'',y''}}{m_{x',y'}}>\frac{m_{x',y'}}{m_{x,y}}$. Thus
$$\begin{array}{rcl}
m_{x,y}+m_{x'',y''}&=&(\frac{m_{x,y}}{m_{x',y'}}+\frac{m_{x'',y''}}{m_{x',y'}})m_{x',y'}\vspace{1.5mm}\\
&> &
(\frac{m_{x,y}}{m_{x',y'}}+\frac{m_{x',y'}}{m_{x,y}})m_{x',y'} \vspace{1.5mm}\\
&\geq &
2m_{x',y'}.
\end{array}$$

If $b=0$ then $k\geq 0$. By Corollaries \ref{cor-h1} and \ref{cor-v1}, we have $r_l(x,y)>1$ for all $(x,y)\in l\cap \mathbb Z_{\geq 0}^2$. Assume that $q,p$ are corpime and $(q,p)\in l$. By Lemma \ref{lem-fn}, we have $lim_{x\to \infty} r_l(x,y)+\frac{1}{lim_{x\to \infty} r_l(x,y)}=3m_{q,p}>3$. It follows that $lim_{x\to \infty} r_l(x,y)>2$. By Proposition \ref{prop-l1}, $r_l(x,y)$ is decreasing along $x$, it follows that $r_l(x_{i},y_{i})>lim_{x\to \infty} r_l(x,y)>2$ for any integral point $(x_i,y_i)$ on $l$. It follows that $\frac{m_{x'',y''}}{m_{x',y'}}>2$.
Thus
$$m_{x,y}+m_{x'',y''}>2m_{x',y'}.$$

\end{proof}

The following result generalizes Lemmas \ref{lem-h1} and \ref{lem-v1} to the lines with a negative slope.

\begin{proposition}\label{prop-l2}
Let $l,l'$ be two lines with same negative slope. Let $(x_i,y_i)$ and $(x'_i,y'_i)$ be integral  points on $l$ and $l'$, respectively.
\begin{enumerate}[$(1)$]
\item If $x_i<x'_i$ and $y_i=y'_i$ then
$$r_l(x_i,y_i) < r_{l'}(x'_i,y'_i).$$
\item If $x_i=x'_i$ and $y_i<y'_i$ then
$$r_l(x_i,y_i) > r_{l'}(x'_i,y'_i).$$
\item If $y_i-x_i=y'_i-x'_i$ and $x_i<x'_i$, then
$$r_l(x_i,y_i) > r_{l'}(x'_i,y'_i).$$
\end{enumerate}
\end{proposition}

\begin{proof}
Let $O=(0,0), A_1=(x_i,y_i), A_2=(x_{i+1},y_{i+1}), A_1'=(x'_i,y'_i), A_2'=(x'_{i+1},y'_{i+1})$. Let $O'=(x_2-x_1,y_2-y_1)$.

(1) As $x_i<x'_i$ and $y_i=y'_i$, $OA_2'$ crosses $O'A_2$. By the Ptolemy inequality, we have $|OA'_2||O'A_2|>|OA_2||O'A_2'|$, that is
$$|OA'_2||OA_1|>|OA_2||OA'_1|.$$
It follows that $\frac{|OA'_2|}{|OA'_1|}>\frac{|OA_2|}{|OA_1|},$ that is
$$r_l(x_i,y_i) < r_{l'}(x'_i,y'_i).$$

(2) As $x_i=x'_i$ and $y_i<y'_i$, $OA_2$ crosses $O'A'_2$. By the Ptolemy inequality, we have $|OA_2||O'A'_2|>|OA'_2||O'A_2|$, that is
$$|OA_2||OA'_1|>|OA'_2||OA_1|.$$
It follows that $\frac{|OA_2|}{|OA_1|}>\frac{|OA'_2|}{|OA'_1|},$ that is
$$r_l(x_i,y_i) > r_{l'}(x'_i,y'_i).$$

(3) As $x_i<x'_i$ and $OO'=A_1A_2=A'_1A'_2$, we have $OA_2$ crosses $O'A'_2$. By the Ptolemy inequality, we have $|OA_2||O'A'_2|>|OA'_2||O'A_2|$, that is
$$|OA_2||OA'_1|>|OA'_2||OA_1|.$$
It follows that $\frac{|OA_2|}{|OA_1|}>\frac{|OA'_2|}{|OA'_1|},$ that is
$$r_l(x_i,y_i) > r_{l'}(x'_i,y'_i).$$

\begin{figure}[h]
\includegraphics{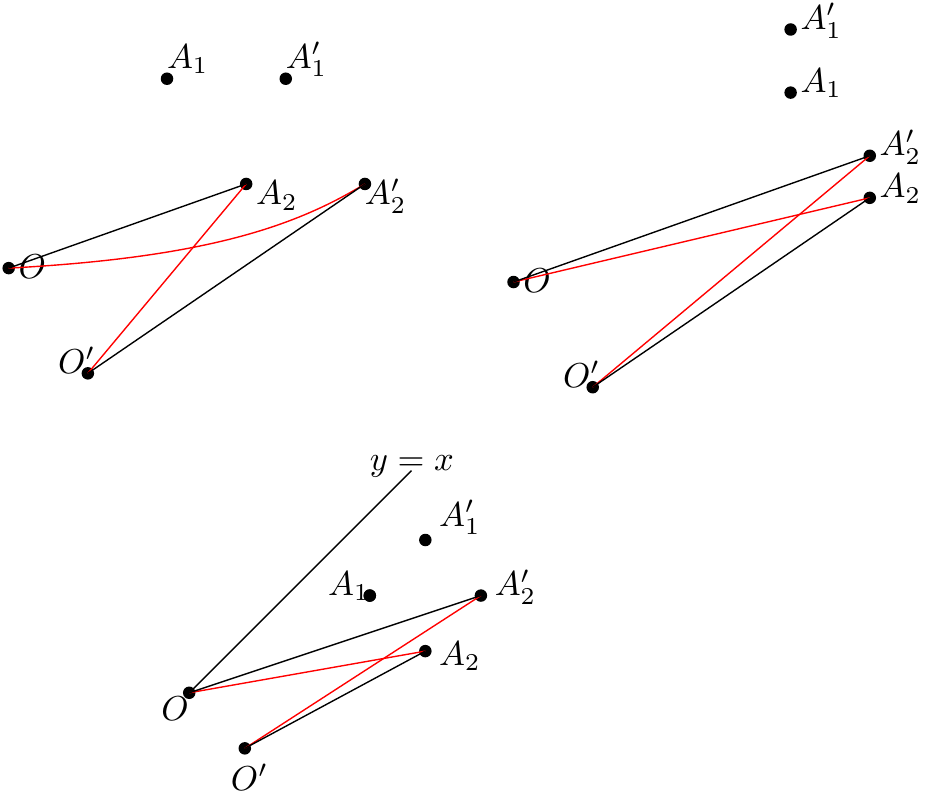}

{\rm Figure for Proposition \ref{prop-l2}}
\end{figure}

\end{proof}

Proposition \ref{prop-l2} implies the ratios along the lines with a given negative slope are increase with $x$, decrease with $y$.

The following follows immediately from Proposition \ref{prop-l2}.

\begin{corollary}\label{cor-r1}
With the foregoing notation. Let $l$ be a line with a negative slope. Then the sequence $r_{l[t]}(z_{l[t]},w_{l[t]})$, $t\in \mathbb Z_{>0}$ is a strictly decreasing sequence, where $(z_{l[t]},w_{l[t]})$ is given in Notation \ref{no-1}.
\end{corollary}

\section{Monotonicity of the generalized Markov numbers}\label{sec-4}

\subsection{On the last ratio}

Recall that we denote by $(z_l, w_l), (z'_l, w'_l)$ the last two integral  points on $l$ for any line $l$ with a negative slope. In this subsection, we study the ratio $r_l(z_l,w_l)$.

For a given $k=-\frac{a_1}{a_2}\in \mathbb Q_{<0}$ with $a_1,a_2\in \mathbb Z_{>0}$ and $g.c.d.(a_1,a_2)=1$, denote by $l_n: y=k(x-n)+1, n\in \mathbb Z$ the sequence of lines through $(n,1)$ with slope $k$. The last two integral  points on $l_n$ are $(n-a_2,1+a_1)$ and $(n,1)$. For any line $l: y=kx+b$ with slope $k$, assume that $l\neq l_n$ for any $n\in \mathbb Z$, then there exists $N^+\in \mathbb Z$ such that $l$ lies between $l_{N^+-1}$ and $l_{N^+}$. Thus $z_l<N^+-1$. Note that to ensure there are at least two integral  points in $l\cap \{(x,y)\in \mathbb Z_{>0}^2\mid x>y\}$, we may assume that $N^+\geq 4$.

\begin{lemma}\label{lem-r1}
With the foregoing notation. We have
$$r_l(z_l,w_l)< r_{l_{N^+}}(z_{l_{N^+}},w_{l_{N^+}}).$$
\end{lemma}

\begin{proof}
Let $O=(0,0), A_1=(z_l,w_l), A_2=(z'_l,w'_l), A_1'=(N^+-a_2,1+a_1), A_2'=(N^+,1)$. As the slopes of $l$ and $l_n$ are $k$, $A_1,A_2,A'_2$ and $A'_1$ form a parallelogram. Let $O'=(a_2,-a_1)$.

As $z_l,a_2<N^+$ and $w_l>1$, we see that $OA_2'$ crosses $O'A_2$. By the Ptolemy inequality, we have $|OA'_2||O'A_2|>|OA_2||O'A_2'|$, that is
$$|OA'_2||OA_1|>|OA_2||OA'_1|.$$
It follows that $\frac{|OA'_2|}{|OA'_1|}>\frac{|OA_2|}{|OA_1|},$ that is
$$r_l(z_l,w_l)< r_{l_{N^+}}(z_{l_{N^+}},w_{l_{N^+}}).$$

\begin{figure}[h]
\includegraphics{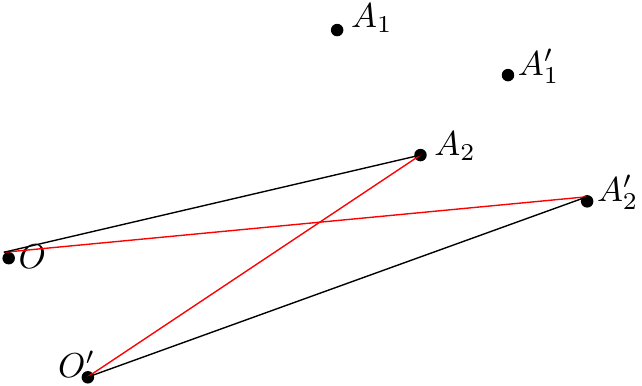}

{\rm Figure for Lemma \ref{lem-r1}}
\end{figure}

\end{proof}

\begin{lemma}\label{lem-r2}
Let $l$ be a line with negative slope $k=-\frac{a_1}{a_2}$ with $a_1,a_2\in \mathbb Z_{>0}$ and $g.c.d.(a_1,a_2)=1$. For a larger enough $t\in \mathbb Z_{>0}$, there exists $N_-\in \mathbb N$ such that
$$r_{l[t]}(z_{l[t]},w_{l[t]})> r_{l_{N^-}}(z_{l_{N^-}},w_{l_{N^-}}),$$
moreover, $N^-$ can be chosen such that $lim_{t\to +\infty} N^-=+\infty$.
\end{lemma}

\begin{proof}
Let $O=(0,0), O'=(a_2,-a_1), A_1=(z_l,w_l), A_2=(z'_l,w'_l)$. For any $t\in \mathbb Z_{>0}$, denote $A_1[t]=(u_l+t,v_l)$ and $A_2[t]=(u'_l+t,v'_l)$. According to Lemma \ref{lem-fl}, $A_1[t]$ and $A_2[t]$ are the last two integral  points on $l[t]$. We see that $A_2[t]$ lies on the line $y=w'_l$. On the other hand, the last integral  points on $l_n, n\in \mathbb Z_{>0}$ lies on the line $y=1$.

Note that if $t$ approximates to $+\infty$, the $x$-coordinate of the crossing point of the segment $O'A_2[t]$ and the line $y=1$ approximate to $+\infty$. Therefore, for larger enough $t$, we can find $N_-\in \mathbb N$ such that $OA_2[t]$ crosses $O'A'_2$, where $A'_2$ is the last integral  points on $l_{N_-}$.

By the Ptolemy inequality, we have $|OA_2[t]||O'A'_2|>|OA'_2||O'A_2[t]|$, that is
$$|OA_2[t]||OA'_1|>|OA'_2||OA_1[t]|.$$
It follows that $\frac{|OA_2[t]|}{|OA_1[t]|}>\frac{|OA'_2|}{|OA'_1|},$ that is
$$r_{l[t]}(z_{l[t]},w_{l[t]})> r_{l_{N^-}}(z_{l_{N^-}},w_{l_{N^-}}).$$

Moreover, when $t$ approximates to $+\infty$, the $x$-coordinate of the crossing point of the segment $O'A_2[t]$ and the line $y=1$ approximates to $+\infty$, so we can choose $N_-\in \mathbb N$ such that $N_-$ approximates to $+\infty$.

\begin{figure}[h]
\includegraphics{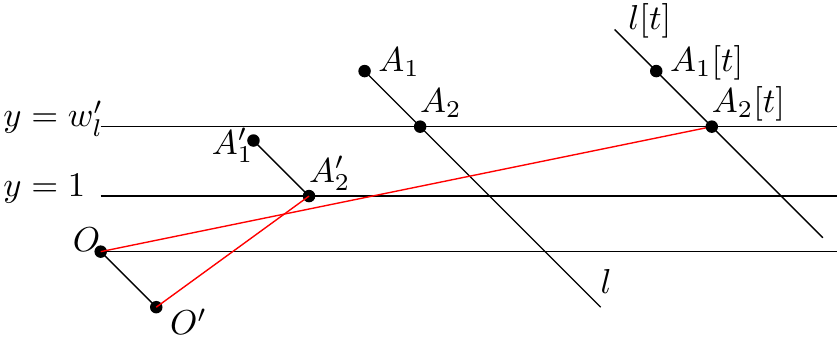}

{\rm Figure for Lemma \ref{lem-r2}}
\end{figure}

\end{proof}

From the proof of Lemma \ref{lem-r2}, we have the following observation.

\begin{lemma}\label{lem-r5}
Assume that $k=-\frac{a_1}{a_2}$ with $a_1,a_2\in \mathbb Z_{>0}$ and $g.c.d.(a_1,a_2)=1$. Let $\gamma:y=kx$ be the line passing through $O$ with slope $k$. Suppose that $n>1+a_1+a_2$
\begin{enumerate}[$(1)$]
  \item Let $l$ be a line with slope $k$. Then there exists $\delta(n,l)\in \mathbb R$ such that
        $$r_{l[t]}(z_{l[t]},w_{l[t]})> r_{l_n}(z_{l_{n}},w_{l_{n}})$$
        for all $t\in \mathbb Z_{>0}$ such that $l[t]$ lies on the right side of $\gamma[\delta(n,l)]$.
  \item Let $l'$ be another line with slope $k$. If $w'_l<w'_{l'}$ then $\gamma[\delta(n,l)]<\gamma[\delta(n,l')]$.
  \item There exists $\delta(n)\in \mathbb R$ such that for any line $l$ with slope $k$,
        $$r_{l[t]}(z_{l[t]},w_{l[t]})> r_{l_n}(z_{l_{n}},w_{l_{n}})$$
        for all $t\in \mathbb Z_{>0}$ such that $l[t]$ lies on the right side of $\gamma[\delta(n)]$.
\end{enumerate}
\end{lemma}

\begin{proof}
(1) Let $O=(0,0), O'=(a_2,-a_1), A'_1=(n-a_1,a_2+1), A'_2=(n,1)$. Assume that the line connecting $O$ and $A'_2$ crosses $y=w'_l$ at some point $A$. Then the line through $A$ with slope $k$ equals to $l[\delta(n,l)]$ for some $\delta(n,l)\in \mathbb R$. Thus, for any $t\in \mathbb Z_{>0}$ such that $l[t]$ lies on the right side of $\gamma[\delta(n,l)]$, we have $OA_2[t]$ crosses $O'A'_2$, where $A_2[t]$ is the last integral  points on $l[t]$. From the proof of Lemma \ref{lem-r5}, we have
$$r_{l[t]}(z_{l[t]},w_{l[t]})> r_{l_n}(z_{l_{n}},w_{l_{n}}).$$

(2) It follows immediately from the construction of $\delta(n,l)$.

(3) From the construction of $\delta(n,l)$, we see that $\delta(n,l)$ only depends on $w'_l$. Note that for any $l$ we have $w'_l\leq a_2$ as $(z'_l,w'_l)$ is the last integral point. Thus we may let $\delta(n)=\delta(n,l)$, where $l$ is the line with slope $k$ such that $w'_l$ is maximal.

\begin{figure}[h]
\includegraphics{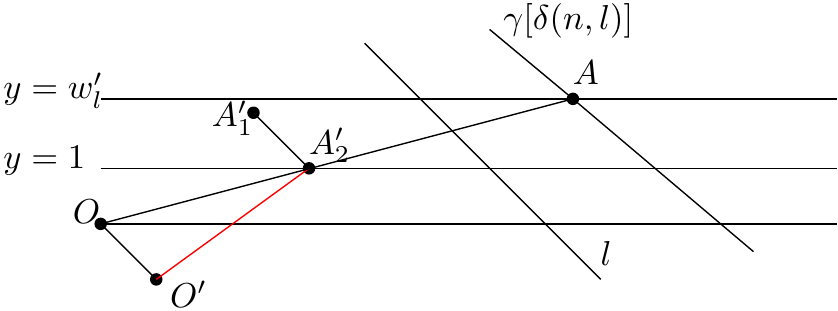}

{\rm Figure for Lemma \ref{lem-r5}}
\end{figure}

\end{proof}

As a corollary of Lemmas \ref{lem-r1} and \ref{lem-r2}, we have the following.

\begin{corollary}\label{cor-r1}
Let $k\in \mathbb Q_{<0}$ and $l:y=kx+b$ be a line with slope $k$. Let $l_n: y=k(x-n)+1$ be the lines through $(n,1)$ with slope $k$. Then
$$lim_{t\to \infty} r_{l[t]}(z_{l[t]},w_{l[t]})=lim_{n\to \infty} r_{l_n}(z_{l_n},w_{l_n}).$$
\end{corollary}

\begin{proposition}\label{prop-lim1}
With the foregoing notation. Assume that $k=-\frac{a_1}{a_2}\in \mathbb Q_{<0}$ with $a_1,a_2\in \mathbb Z_{>0}$. Let $l$ be a line with slope $k$. Then
\begin{enumerate}[$(1)$]
  \item $$lim_{n\to \infty} r_{l_n}(z_{l_n},w_{l_n})=(\frac{2\sqrt 5}{3(1+\sqrt 5)})^{a_1}(\frac{1+\sqrt 5}{2})^{2a_2}.$$
  \item $$lim_{t\to \infty} r_{l[t]}(z_{l[t]},w_{l[t]})=(\frac{2\sqrt 5}{3(1+\sqrt 5)})^{a_1}(\frac{1+\sqrt 5}{2})^{2a_2}.$$
\end{enumerate}

\end{proposition}

\begin{proof}
We shall only prove (1) as (2) follows by Corollary \ref{cor-r1}.

For each $n$, we have $(z_{l_n},w_{l_n})=(n-a_2,1+a_1)$ and $(z'_{l_n},w'_{l_n})=(n,1)$. By Proposition \ref{prop-l2} (1), the sequence $r_{l_n}(z_{l_n},w_{l_n}),n\in \mathbb Z_{>0}$ is strictly increasing. It suffices to consider the subsequence indexed by $q(1+a_1)+a_2, q\in \mathbb Z_{>0}$. By Proposition \ref{prop-app}, we have
$$\begin{array}{rcl}
&& lim_{q\to \infty} r_{l_{q(1+a_1)+a_2}}(z_{l_{q(1+a_1)+a_2}},w_{l_{q(1+a_1)+a_2}})\vspace{1.5mm}\\
&=&\frac{m_{q(1+a_1)+a_2,1}}{m_{q(1+a_1),1+a_1}}\vspace{1.5mm}\\
&=&
\frac{\frac{((1+\sqrt{5})/2)^{2q(1+a_1)+2a_2+1}}{\sqrt{5}}}{3^{a_1}m_{q,1}^{1+a_1}} \vspace{1.5mm}\\
&= &
\frac{\frac{((1+\sqrt{5})/2)^{2q(1+a_1)+2a_2+1}}{\sqrt{5}}}{3^{a_1}(\frac{((1+\sqrt{5})/2)^{2q+1}}{\sqrt{5}})^{1+a_1}}\vspace{1.5mm}\\
&= & (\frac{\sqrt 5}{3})^{a_1}(\frac{1+\sqrt 5}{2})^{-a_1+2a_2}\vspace{1.5mm}\\
&= & (\frac{2\sqrt 5}{3(1+\sqrt 5)})^{a_1}(\frac{1+\sqrt 5}{2})^{2a_2}.
\end{array}$$

\end{proof}

\subsection{On the first ratio}
Recall that we denote by $(u_l, v_l), (u'_l, v'_l)$ the first two integral  points on $l$ for any line $l$ with a negative slope. In this subsection, we study the ratio $r_l(u_l,v_l)$.

For a given $k=-\frac{a_1}{a_2}\in \mathbb Q_{<0}$ with $a_1,a_2\in \mathbb Z_{>0}$ and $g.c.d.(a_1,a_2)=1$, denote by $L_n: y=k(x-n)+n-1, n\in \mathbb Z$ the sequence of lines through $(n,n-1)$ with slope $k$. The first two integral  points on $L_n$ are $(n,n-1)$ and $(n+a_2,n-1-a_1)$. For any line $l: y=kx+b$ with slope $k$, assume that $l\neq L_n$ for any $n\in \mathbb Z$, there exists $\widetilde N_+\in \mathbb Z$ such that $l$ lies between $l_{\widetilde N_+-1}$ and $l_{\widetilde N_+}$.


\begin{lemma}\label{lem-r3}
With the foregoing notation. We have
$$r_l(u_l,v_l)> r_{L_{\widetilde N_+}}(u_{L_{\widetilde N_+}},v_{L_{\widetilde N_+}}).$$
\end{lemma}

\begin{proof}
Let $O=(0,0), A_1=(u_l,v_l), A_2=(u'_l,v'_l), A_1'=(u_{L_{\widetilde N_+}},v_{L_{\widetilde N_+}}), A_2'=(u'_{L_{\widetilde N_+}},v'_{L_{\widetilde N_+}})$. As the slopes of $l$ and $l_n$ are $k$, $A_1,A_2,A'_2$ and $A'_1$ form a parallelogram. Let $O'=(a_2,-a_1)$.

As $l$ lies between $l_{\widetilde N_+-1}$ and $l_{\widetilde N_+}$, we see that $O'A'_2$ crosses $OA_2$. By the Ptolemy inequality, we have $|OA_2||O'A_2|>|OA'_2||O'A_2|$, that is
$$|OA_2||OA'_1|>|OA'_2||OA_1|.$$
It follows that $\frac{|OA_2|}{|OA_1|}>\frac{|OA'_2|}{|OA'_1|},$ that is
$$r_l(u_l,v_l)> r_{L_{\widetilde N_+}}(u_{L_{\widetilde N_+}},v_{L_{\widetilde N_+}}).$$

\begin{figure}[h]
\includegraphics{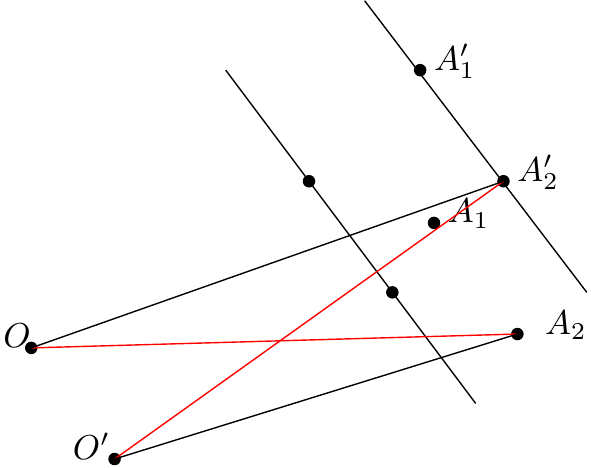}

{\rm Figure for Lemma \ref{lem-r3}}
\end{figure}

\end{proof}

\begin{lemma}\label{lem-r4}
Let $l$ be a line with negative slope $k=-\frac{a_1}{a_2}$ with $a_1,a_2\in \mathbb Z_{>0}$ and $g.c.d.(a_1,a_2)=1$. For a larger enough $t\in \mathbb Z_{>0}$, there exists $\widetilde N_-\in \mathbb N$ such that
$$r_{l\langle t\rangle}(u_{l\langle t\rangle},v_{l\langle t\rangle})< r_{L_{\widetilde N_-}}(u_{L_{\widetilde N_-}},v_{L_{\widetilde N_-}}).$$
moreover, $\widetilde N_-$ can be chosen such that $lim_{t\to \infty} \widetilde N_-=+\infty$.

\end{lemma}

\begin{proof}
Let $O=(0,0), O'=(a_2,-a_1), A_1=(u_l,v_l), A_2=(u'_l,v'_l)$. For any $t\in \mathbb Z_{>0}$, denote $A_1\langle t\rangle=(u_l+t,v_l+t)$ and $A_2\langle t\rangle=(u'_l+t,v'_l+t)$. According to Lemma \ref{lem-fl}, $A_1\langle t\rangle$ and $A_2\langle t\rangle$ are the first two integral  points on $l\langle t\rangle$. We see that $A_2\langle t\rangle$ lies on the line $y=x-u'_l+v'_l$. On the other hand, the second integral  points on $L_n, n\in \mathbb Z_{>0}$ lies on the line $y=x-a_1-a_2-1$.

Note that if $t$ approximates to $+\infty$, the $x$-coordinate of the crossing point of the segment $O'A_2\langle t\rangle $ and the line $y=x-a_1-a_2-1$ approximate to $+\infty$. Therefore, for larger enough $t$, we can find $\widetilde N_-\in \mathbb N$ such that $OA'_2$ crosses $O'A_2\langle t\rangle $, where $A'_2$ is the second integral points on $L_{\widetilde N_-}$.

By the Ptolemy inequality, we have $|O'A_2\langle t\rangle||OA'_2|>|OA_2\langle t\rangle||O'A'_2|$, that is
$$|OA_1\langle t\rangle||OA'_2|>|OA_2\langle t\rangle||OA'_1|.$$
It follows that $\frac{|OA'_2|}{|OA'_1|}>\frac{|OA_2\langle t\rangle|}{|OA_1\langle t\rangle|},$ that is
$$r_{l\langle t\rangle}(u_{l\langle t\rangle},v_{l\langle t\rangle})< r_{L_{\widetilde N_-}}(u_{L_{\widetilde N_-}},v_{L_{\widetilde N_-}}).$$

Moreover, when $t$ approximates to $+\infty$, the $x$-coordinate of the crossing point of the segment $O'A_2\langle t\rangle $ and the line $y=x-a_1-a_2-1$ approximates to $+\infty$, so we can choose $\widetilde N_-\in \mathbb N$ such that $\widetilde N_-$ approximates to $+\infty$.

\begin{figure}[h]
\includegraphics{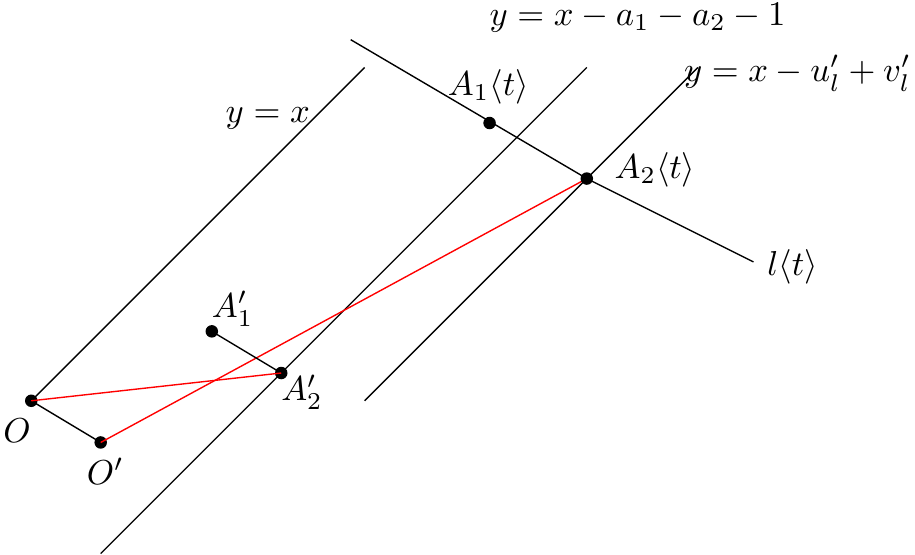}

{\rm Figure for Lemma \ref{lem-r4}}
\end{figure}

\end{proof}

\begin{lemma}\label{lem-r6}
Assume that $k=-\frac{a_1}{a_2}$ with $a_1,a_2\in \mathbb Z_{>0}$ and $g.c.d.(a_1,a_2)=1$. Let $\gamma:y=kx$ be the line passing through $O$ with slope $k$. Suppose that $n>1+a_1$.
\begin{enumerate}[$(1)$]
  \item Let $l$ be a line with slope $k$. Then there exists $\sigma(n,l)\in \mathbb R$ such that
        $$r_{l\langle t\rangle}(u_{l\langle t\rangle},v_{l\langle t\rangle})< r_{L_n}(u_{L_{n}},v_{L_{n}})$$
        for all $t\in \mathbb Z_{>0}$ such that $l\langle t\rangle$ lies on the right side of $\gamma\langle\sigma(n,l)\rangle$.
  \item Let $l'$ be another line with slope $k$. If $u'_l-v'_l<u'_{l'}-v'_{l'}$ then $\sigma(n,l)<\sigma(n,l')$.
  \item There exists $\sigma(n)\in \mathbb R$ such that for any line $l$ with slope $k$,
        $$r_{l\langle t\rangle}(u_{l\langle t\rangle},v_{l\langle t\rangle})< r_{L_n}(u_{L_{n}},v_{L_{n}})$$
        for all $t\in \mathbb Z_{>0}$ such that $l\langle t\rangle$ lies on the right side of $\gamma\langle\sigma(n)\rangle$.
\end{enumerate}
\end{lemma}

\begin{proof}
(1) Let $O=(0,0), O'=(a_2,-a_1), A'_1=(n,n-1), A'_2=(n+a_2,n-1-a_1)$. Assume that the line connecting $O$ and $A'_2$ crosses $y=x-u'_l+v'_l$ at some point $A$. Then the line through $A$ with slope $k$ equals to $l\langle\sigma(n,l)\rangle$ for some $\sigma(n,l)\in \mathbb R$. Thus, for any $t\in \mathbb Z_{>0}$ such that $l\langle t\rangle$ lies on the right side of $\gamma\langle\sigma(n,l)\rangle$, we have $OA_2'$ crosses $O'A_2\langle t\rangle$, where $A_2\langle t\rangle$ is the second integral  points on $l\langle t\rangle$. From the proof of Lemma \ref{lem-r5}, we have
$$r_{l\langle t\rangle}(u_{l\langle t\rangle},v_{l\langle t\rangle})< r_{L_n}(u_{L_{n}},v_{L_{n}})$$

(2) It follows immediately from the construction of $\sigma(n,l)$.

(3) From the construction of $\sigma(n,l)$, we see that $\sigma(n,l)$ only depends on $u'_l-v'_l$. Note that for any $l$ we have $u'_l-v'_l\leq 2a_1+1$ as $(u'_l,v'_l)$ is the second integral  point. Thus we may let $\sigma(n)=\sigma(n,l)$, where $l$ is the line with slope $k$ such that $u'_l-v'_l$ is maximal.

\begin{figure}[h]
\includegraphics{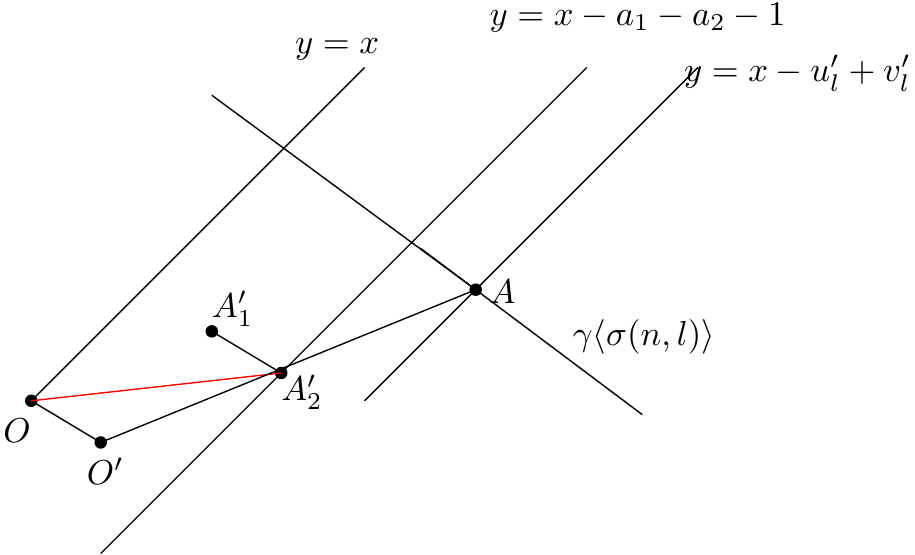}

{\rm Figure for Lemma \ref{lem-r6}}
\end{figure}

\end{proof}

As a corollary of Lemmas \ref{lem-r3} and \ref{lem-r4}, we have the following.

\begin{corollary}\label{cor-r2}
With the foregoing notation. Let $l$ be a line with a negative slope. Then
$$lim_{t\to \infty} r_{l\langle t\rangle}(u_{l\langle t\rangle},v_{l\langle t\rangle})=lim_{n\to \infty} r_{L_n}(u_{L_n},v_{L_n}).$$
\end{corollary}

\begin{proposition}\label{prop-lim2}
With the foregoing notation. Assume that $k=-\frac{a_1}{a_2}\in \mathbb Q_{<0}$ with $a_1,a_2\in \mathbb Z_{>0}$ and $g.c.d.(a_1,a_2)=1$. Let $l$ be a line with slope $k$. Then
\begin{enumerate}[$(1)$]
  \item $$lim_{n\to \infty} r_{L_n}(u_{L_n},v_{L_n})=(\frac{3}{2\sqrt 2})^{a_1+a_2}(1+\sqrt 2)^{-a_1+a_2}.$$
  \item $$lim_{t\to \infty} r_{l\langle t\rangle}(u_{l\langle t\rangle},v_{l\langle t\rangle})=(\frac{3}{2\sqrt 2})^{a_1+a_2}(1+\sqrt 2)^{-a_1+a_2}.$$
\end{enumerate}

\end{proposition}

\begin{proof}
We shall only prove (1) as (2) follows by Corollary \ref{cor-r2}.

For each $n$, we have $(u_{L_n},v_{L_n})=(n,n-1)$ and $(u'_{L_n},v'_{L_n})=(n+a_2,n-1-a_1)$. By Proposition \ref{prop-l2} (3), the sequence $r_{L_n}(u_{L_n},v_{L_n}),n\in \mathbb Z_{>0}$ is strictly decreasing. It suffices to consider the subsequence indexed by $q(1+a_1+a_2)+1+a_1, q\in \mathbb Z_{>0}$. By Proposition \ref{prop-app}, we have
$$\begin{array}{rcl}
&& lim_{q\to \infty} r_{L_{q(1+a_1+a_2)+1+a_1}}(u_{L_{q(1+a_1+a_2)+1+a_1}},v_{L_{q(1+a_1+a_2)+1+a_1}})\vspace{1.5mm}\\
&=&\frac{m_{(q+1)(1+a_1+a_2),q(1+a_1+a_2)}}{m_{q(1+a_1+a_2)+1+a_1,q(1+a_1+a_2)+1+a_1-1}}\vspace{1.5mm}\\
&=&
\frac{3^{a_1+a_2}(m_{q+1,q})^{1+a_1+a_2}}{\frac{((1+\sqrt 2)^{2q(1+a_1+a_2)+2a_1+1}}{2\sqrt{2}}} \vspace{1.5mm}\\
&= &
\frac{3^{a_1+a_2}(\frac{(1+\sqrt{2})^{2q+1}}{2\sqrt{2}})^{1+a_1+a_2}}{\frac{((1+\sqrt 2)^{2q(1+a_1+a_2)+2a_1+1}}{2\sqrt{2}}}\vspace{1.5mm}\\
&= & (\frac{3}{2\sqrt 2})^{a_1+a_2}(1+\sqrt 2)^{-a_1+a_2}.
\end{array}$$

\end{proof}

\subsection{On the monotonicity}

\begin{theorem}\label{thm-mono}
\begin{enumerate}[$(1)$]
  \item For $k\in \mathbb Q$ with $k\geq -\frac{ln\frac{3(2+\sqrt 2)}{4}}{ln\frac{2(2+\sqrt 2)}{3}}\approx -1.1432$, the generalized Markov numbers increase with $x$ along any line $l: y=kx+b$;
  \item For $k\in \mathbb Q$ with $k\leq -\frac{2ln(\frac{1+\sqrt{5}}{2})}{ln(\frac{3(1+\sqrt{5})}{2\sqrt{5}})}\approx -1.2417$, the generalized Markov numbers decrease with $x$ along any line $y=kx+b$;
  \item For any $k\in \mathbb Q$ with $-\frac{2ln(\frac{1+\sqrt{5}}{2})}{ln(\frac{3(1+\sqrt{5})}{2\sqrt{5}})}<k<-\frac{ln\frac{3(2+\sqrt 2)}{4}}{ln\frac{2(2+\sqrt 2)}{3}}$, then for almost all $b\in \mathbb Q$, the generalized Markov numbers are not monotonic along the line $y=kx+b$.
\end{enumerate}

\end{theorem}

\begin{proof}
If $a_1\leq a_2$ then we have $(\frac{3}{2\sqrt 2})^{a_1+a_2}(1+\sqrt 2)^{-a_1+a_2}>1$. If $a_1<a_2$, then $(\frac{3}{2\sqrt 2})^{a_1+a_2}(1+\sqrt 2)^{-a_1+a_2}\geq 1$ if and only if $(\frac{3}{2\sqrt 2})^{a_1+a_2}\geq (1+\sqrt 2)^{a_1-a_2}\Longleftrightarrow(a_1+a_2)ln(\frac{3}{2\sqrt 2})\geq (a_1-a_2)ln(1+\sqrt 2)\Longleftrightarrow (ln(\frac{3}{2\sqrt 2})+ln(1+\sqrt 2))a_2\geq (ln(1+\sqrt 2)-ln(\frac{3}{2\sqrt 2}))a_1\Longleftrightarrow \frac{a_1}{a_2}\leq \frac{ln(1+\sqrt{2})+ln(3\sqrt{2}/4)}{ln(1+\sqrt{2})-ln(3\sqrt{2}/4)}$. As $\frac{ln(1+\sqrt{2})+ln(3\sqrt{2}/4)}{ln(1+\sqrt{2})-ln(3\sqrt{2}/4)}\geq 1$, we have $(\frac{3}{2\sqrt 2})^{a_1+a_2}(1+\sqrt 2)^{-a_1+a_2}\geq 1$ if and only if $\frac{a_1}{a_2}\leq \frac{ln(1+\sqrt{2})+ln(3\sqrt{2}/4)}{ln(1+\sqrt{2})-ln(3\sqrt{2}/4)}$ if and only if $-\frac{a_1}{a_2}\geq -\frac{ln(1+\sqrt{2})+ln(3\sqrt{2}/4)}{ln(1+\sqrt{2})-ln(3\sqrt{2}/4)}=-\frac{ln\frac{3(2+\sqrt 2)}{4}}{ln\frac{2(2+\sqrt 2)}{3}}$.

\smallskip

On the other hand, $(\frac{2\sqrt 5}{3(1+\sqrt 5)})^{a_1}(\frac{1+\sqrt 5}{2})^{2a_2}\leq 1\Longleftrightarrow (\frac{1+\sqrt 5}{2})^{2a_2}\leq (\frac{3(1+\sqrt 5)}{2\sqrt 5})^{a_1}\Longleftrightarrow 2a_2ln(\frac{1+\sqrt 5}{2})\leq a_1ln(\frac{3(1+\sqrt 5)}{2\sqrt 5})\Longleftrightarrow \frac{a_1}{a_2}\geq \frac{2ln(\frac{1+\sqrt{5}}{2})}{ln(\frac{3(1+\sqrt{5})}{2\sqrt{5}})}\Longleftrightarrow -\frac{a_1}{a_2}\leq -\frac{2ln(\frac{1+\sqrt{5}}{2})}{ln(\frac{3(1+\sqrt{5})}{2\sqrt{5}})}$.

(1) For any $l$, by Proposition \ref{prop-l2} (3), the sequence $r_{l\langle t\rangle}(u_{l\langle t\rangle},v_{l\langle t\rangle}), t\in \mathbb Z_{>0}$ is decreasing. If $k\geq -\frac{lg(1+\sqrt{2})+lg(3\sqrt{2}/4)}{lg(1+\sqrt{2})-lg(3\sqrt{2}/4)}$, we have $(\frac{3}{2\sqrt 2})^{a_1+a_2}(1+\sqrt 2)^{-a_1+a_2}>1$. Thus by proposition \ref{prop-lim2}, we see that $r_{l}(u_{l},v_{l})>1$. Then the result follows by Proposition \ref{prop-mon}.

(2) The proof is similar to the proof of (1).

(3) If $-\frac{2ln(\frac{1+\sqrt{5}}{2})}{ln(\frac{3(1+\sqrt{5})}{2\sqrt{5}})}<k<-\frac{ln\frac{3(2+\sqrt 2)}{4}}{ln\frac{2(2+\sqrt 2)}{3}}$, assume that $k=-\frac{a_1}{a_2}$ with $a_1,a_2\in \mathbb Z_{>0}$ and $g.c.d.(a_1,a_2)=1$, then we have $(\frac{3}{2\sqrt 2})^{a_1+a_2}(1+\sqrt 2)^{-a_1+a_2}<1$ and $(\frac{2\sqrt 5}{3(1+\sqrt 5)})^{a_1}(\frac{1+\sqrt 5}{2})^{2a_2}>1$.

As $(\frac{3}{2\sqrt 2})^{a_1+a_2}(1+\sqrt 2)^{-a_1+a_2}<1$, by Proposition \ref{prop-lim2} (1), there exists $n$ such that $r_{L_n}(u_{L_n},v_{L_n})<1$. By Lemma \ref{lem-r6}, there exists $\sigma(n)$ such that for any line $l$ with slope $k$, we have
 $r_{l\langle t\rangle}(u_{l\langle t\rangle},v_{l\langle t\rangle})< r_{L_n}(u_{L_{n}},v_{L_{n}})<1$ for all $t\in \mathbb Z_{>0}$ such that $l\langle t\rangle$ lies on the right side of $\gamma\langle\sigma(n)\rangle\gamma[\sigma(n)-\frac{\sigma(n)}{k}]$.

As $(\frac{2\sqrt 5}{3(1+\sqrt 5)})^{a_1}(\frac{1+\sqrt 5}{2})^{2a_2}>1$, by Proposition \ref{prop-lim1} (1), there exists $n'$ such that $r_{l_{n'}}(z_{l_{n'}},w_{l_{n'}})>1$. By Lemma \ref{lem-r5}, there exists $\delta(n')\in \mathbb R$ such that for any line $l$ with slope $k$, we have
        $r_{l[t]}(z_{l[t]},w_{l[t]})> r_{l_n}(z_{l_{n}},w_{l_{n}})>1$ for all $t\in \mathbb Z_{>0}$ such that $l[t]$ lies on the right side of $\gamma[\delta(n')]$.

Let $\eta=max\{\sigma(n)-\frac{\sigma(n)}{k}, \delta(n')\}$. Then for any line $l$ with slope $k$ which lies on the right side of $\gamma[\eta]$, we have $r_{l}(u_{l},v_{l})<1$ and $r_{l}(z_{l},w_{l})> 1$. By Proposition \ref{prop-mon}, the generalized Markov numbers neither increase nor decrease with $x$ along $l$.

\end{proof}

In view of \cite[Theorem 1.1]{J} on the monotonicity of the usual Markov numbers and Theorem \ref{thm-mono}, compare with the uniqueness conjecture of Markov numbers, we propose the following uniqueness conjecture for the generalized Markov numbers.

\begin{conjecture}
For any $(q,p),(q',p')\in \mathbb Z_{>0}^2$ with $q\geq p, q'\geq p'$, if $(q,p)\neq (q',p')$ then we have
$$m_{q,p}\neq m_{q',p'}.$$
\end{conjecture}

{\bf Acknowledgements:}\; {\em
The author is indebted to Li Li for his interest and valuable conversations.
This project is supported by the National
Natural Science Foundation of China (No.12101617) (No.12071422), Guangdong Basic and Applied Basic Research Foundation 2021A1515012035, and the Fundamental
Research Funds for the Central Universities, Sun Yat-sen University.}


\begin{thebibliography}{99}

\bibitem{A} M.Aigner, Markov's theorem and 100 years of the uniqueness conjecture, Springer, Cham, 2013. A mathematical journey from irrational numbers to perfect matchings.

\bibitem{BBH} A. Beineke, T. Br\"{u}stle and L. Hille, Cluster-cylic quivers with three vertices and the Markov equation.
With an appendix by Otto Kerner. Algebr. Represent. Theory 14 (2011), no. 1, 97-112.


\bibitem{CL} I. Canakci, P. Lampe, An expansion formula for type A and Kronecker quantum cluster
algebras, J. Combin. Theory Ser. A 171 (2020), 105132, 30 pp.



\bibitem{FST} S. Fomin, M. Shapiro, D. Thurston, Cluster algebras and triangulated surfaces. Part I: Cluster complexes, Acta Math. 201 (2008) 83-146.

\bibitem{F} G. Frobenius, \"{U}ber die Markoffschen Zahlen, S. B. Preuss. Akad. Wiss. Berlin (1913) 458-487; available in Gesammelte Abhandlungen Band III, Springer (1968) 598-627.

\bibitem{J} J. Gaster, Boundary slopes for the Markov ordering on relatively prime pairs, Advances in Mathematics 403 (2022) 108377.


\bibitem{H1} M. Huang, Positivity for quantum cluster algebras from unpunctured orbifolds, arXiv:1810.04359.

\bibitem{H2} M. Huang, New expansion formulas for cluster algebras from surfaces, Journal of Algebra 588 (2021) 538-573.

\bibitem{H3} M. Huang, An expansion formula for quantum cluster algebras from unpunctured triangulated surfaces, Selecta Math. (N.S.) 28 (2022), no. 2, Paper No. 21, 58 pp.

\bibitem{LLRS} K. Lee, L. Li, M. Rabideau, R. Schiffler, On the ordering of the Markov numbers, preprint, arXiv:2010.13010.

\bibitem{MS} G. Musiker, R. Schiffler, Cluster expansion formulas and perfect matchings, J. Algebraic Combin. 32 (2) (2010) 187-209.

\bibitem{MSW} G. Musiker, R. Schiffler, and L. Williams, Positivity for cluster algebras from surfaces, Adv. Math. 227 (2011) 2241-2308.

\bibitem{MSW1} G. Musiker, R. Schiffler, L. Williams, Bases for cluster algebras form surfaces,  Compos.  Math.  149(2) (2013) 217-263

\bibitem{MW} G. Musiker and L. Williams, Matrix formulae and skein relations for cluster algebras from surfaces. Int.
Math. Res. Not. IMRN (2013), no. 13, 2891-2944.

\bibitem{P} J. Propp, The combinatorics of frieze patterns and Markoff numbers, preprint, arXiv:math.CO/0511633.

\bibitem{RS} M. Rabideau, R. Schiffler, Continued fractions and orderings on the Markov numbers, Adv. Math.
370 (2020) 107231.

\end{thebibliography}
\end{document}